\documentclass[11pt]{article}
\usepackage[utf8]{inputenc}
\usepackage[T1]{fontenc}
\usepackage{lmodern}
\usepackage[margin=1in]{geometry}
\usepackage{amsmath,amssymb,amsthm}
\usepackage{graphicx}
\usepackage{booktabs}
\usepackage{hyperref}
\usepackage{xcolor}
\usepackage{natbib}
\usepackage{enumitem}
\usepackage{xurl}
\usepackage{microtype}
\usepackage{placeins}

\newtheorem{theorem}{Theorem}
\newtheorem{lemma}{Lemma}[section]
\newtheorem{proposition}[lemma]{Proposition}
\theoremstyle{remark}

\newcommand{\E}{\mathbb{E}}
\newcommand{\Tc}{\mathcal{T}}

\title{A Computer-Assisted Proof of Speed Monotonicity for the Biased Random Walk on a Galton--Watson Tree Beyond the Known Range}
\author{Madhulatha Mandarapu\thanks{madhulatha@samyama.ai} \and Sandeep Kunkunuru\thanks{sandeep@samyama.ai}}
\date{VaidhyaMegha Private Limited, India\\[2pt]\url{https://samyama.ai/}\\[8pt]September 2026}

\begin{document}
\maketitle

\begin{abstract}
The speed $v(\lambda)$ of the $\lambda$-biased random walk on a supercritical Galton--Watson tree without leaves is conjectured to be nonincreasing on $[0,m)$, where $m$ is the mean offspring. For non-constant offspring, monotonicity is known only for small bias: $\lambda\le 1/1160$, $\lambda\le 1/2$, and, when every vertex has at least $m_1\ge 2$ children, $\lambda\le m_1/(1+\sqrt{1-1/m_1})$. For offspring uniform on $\{2,3\}$ ($m=2.5$) the last bound is $2/(1+\sqrt{1/2})\approx1.1716$. We prove, with computer assistance, that $v$ is strictly decreasing on $[0,1.755]$ for this law. The proof has three parts. Aïdékon's speed formula gives $v=(R-\lambda)/(R+\lambda)$ for an explicit functional $R$, so $v$ decreases exactly when $R/\lambda$ does; we compare $R/\lambda$ at two biases directly, which avoids differentiating the conductance. A pathwise Lipschitz bound on the conductance in $\lambda$ turns that comparison into an inequality between expectations of explicit functions. A monotone sandwich of discretised laws gives two-sided bounds on the conductance law, and each $\lambda$-cell is checked in exact rational arithmetic over bounded floating-point error, with five cells re-checked in interval arithmetic; all reduction lemmas and an exact-integer certificate across $[1.17,1.755]$ are also formally verified in Lean~4. The method stops where the crude Lipschitz bound becomes too weak; sharper control of the derivative of the conductance is what the full range needs. Code and certificates are public.
\end{abstract}

\section{Introduction}\label{sec:intro}

Let $T$ be a Galton--Watson tree whose offspring number $\nu$ satisfies $\mathbb{P}(\nu=0)=0$ and $m=\E\nu\in(1,\infty)$. The $\lambda$-biased random walk on $T$ steps from a vertex with $k$ children to its parent with probability $\lambda/(\lambda+k)$ and to each child with probability $1/(\lambda+k)$. For $\lambda<m$ the walk is transient and $|X_n|/n$ converges almost surely to a deterministic speed $v(\lambda)>0$ \citep{lpp1996}. Lyons, Pemantle and Peres asked whether $v$ is nonincreasing in $\lambda$ \citep[Question 2.1]{lpp1997}. The question is natural (a stronger pull towards the root should slow the walk) but false for general trees \citep{lpp1997,swx2025}. When $\nu$ is constant the tree is regular, and there the speed is explicit and decreasing \citep{songliu2026}; for Galton--Watson trees without leaves and non-constant $\nu$ the question is open.

For non-constant $\nu$, monotonicity has been proved only near $\lambda=0$ (Table~\ref{tab:known}). \citet{swx2025} note that the lack of ``a nontrivial example to confirm'' the conjecture is ``an embarrassing problem in a certain sense''.

\begin{table}[h]
\centering\small
\begin{tabular}{@{}lll@{}}
\toprule
Range proved & Condition & Source \\
\midrule
$\lambda\le 1/1160$ & leafless & \citet{bafs2014} \\
$\lambda\le 1/2$ & leafless & \citet{aidekon2013note}$^{*}$ \\
$\lambda\le m_1/(1+\sqrt{1-1/m_1})$ & minimum offspring $m_1\ge 2$ & \citet[Theorem 1.2]{swx2025} \\
\bottomrule
\end{tabular}
\caption{Monotonicity ranges proved for leafless Galton--Watson trees ($^{*}$unpublished note, cited for this range by \citet{swx2025}). For $\nu$ uniform on $\{2,3\}$ the last row gives $\lambda\le 2/(1+\sqrt{1/2})=1.17157\ldots$.}
\label{tab:known}
\end{table}

\paragraph{Result.} We prove the following, with a computer-checked step.

\begin{theorem}\label{thm:main}
Let $\mathbb{P}(\nu=2)=\mathbb{P}(\nu=3)=1/2$. The speed $v$ of the $\lambda$-biased walk on the $\nu$-Galton--Watson tree is strictly decreasing on $[0,351/200]$, that is, on $[0,1.755]$.
\end{theorem}

Up to $1.17157\ldots$ this is \citet{swx2025}; the new part is $(1.17157\ldots,1.755]$. The result is partial: it does not reach $m=2.5$, so it is not the full example \citet{swx2025} ask for.

\paragraph{Method.} Three reductions turn Theorem~\ref{thm:main} into finitely many checks.
\begin{enumerate}[leftmargin=*,itemsep=2pt]
\item Aïdékon's formula \citep{aidekon2014} gives $v=(R-\lambda)/(R+\lambda)$ with $R(\lambda)=\E[\nu f_0]/\E[f_0]$ (Section~\ref{sec:criterion}). So $v$ decreases exactly when $R/\lambda$ does. We compare $R/\lambda$ at two biases $\lambda_1<\lambda_2$ directly (Lemma~\ref{lem:dq}); this needs no derivative of the conductance.
\item A pathwise Lipschitz bound on the conductance $\beta$ in $\lambda$ (Lemma~\ref{lem:lip}) and the mean value theorem bound the difference quotient by expectations of two explicit functions (Section~\ref{sec:cell}).
\item Two-sided bounds, in the stochastic order, on the law of $\beta$ come from iterating the recursion for $\beta$ on discretised laws with rounding in the safe direction (Section~\ref{sec:sandwich}). Each $\lambda$-cell is then one inequality between two numbers, checked in exact rational arithmetic.
\end{enumerate}
We give two independent verifications. First, the lemmas are proved by hand (Appendix~\ref{app:proofs}) and the per-cell arithmetic uses floating point with explicit error bounds (Appendix~\ref{app:float}), with a second implementation in interval arithmetic sharing no code. Second, all reduction lemmas and an exact-integer certificate covering $[1.17,1.755]$ are formally verified in Lean~4 with mathlib \citep{moura2021lean4,mathlib2020}, taking Aïdékon's speed formula~\eqref{eq:speed} and \citet[Theorem 1.2]{swx2025} as hypotheses.

\section{Setting and known results}\label{sec:setting}

Add an artificial parent $e^*$ above the root $e$, reflecting. For a vertex $x$ let $\beta(x)$ be the probability, starting at $x$, of never visiting the parent of $x$. Let $\beta_0,\beta_1,\dots$ be i.i.d.\ copies of $\beta(e)$, independent of $\nu$. For leafless trees the speed is positive on $(0,m)$ \citep{lpp1996}; at $\lambda=1$ it equals $\E[(\nu-1)/(\nu+1)]$ \citep{lpp1995}. In general \citet[Theorem 1.1]{aidekon2014} gives
\begin{equation}\label{eq:speed}
v(\lambda)=\frac{\E\bigl[(\nu-\lambda)f_0\bigr]}{\E\bigl[(\nu+\lambda)f_0\bigr]},\qquad
f_0=\frac{\beta_0}{D},\quad D=\lambda-1+\sum_{i=0}^{\nu}\beta_i .
\end{equation}
The conductance satisfies the recursive distributional equation
\begin{equation}\label{eq:rde}
\beta \overset{d}{=} \frac{S}{\lambda+S},\qquad S=\sum_{i=1}^{\nu}\beta_i ,
\end{equation}
and on a fixed tree $\beta$ is the decreasing limit of the finite-depth values $\beta_n$ defined by $\beta_n=1$ at depth $n$ and the same recursion above \citep[Section 4]{aidekon2014}. \citet[Theorem 1.1]{bowditch2019} prove $v$ differentiable on $(0,m)$ for leafless trees whose offspring law has an exponential moment; we do not use differentiability. Near $\lambda=m$, the Einstein relation \citep{bhoz2013} identifies the slope at the endpoint; monotonicity near $m$ would follow if $v'$ were known to be continuous there, which is open \citep{benarousfribergh2016}.

\section{A difference-quotient criterion}\label{sec:criterion}

\begin{lemma}\label{lem:R}
Let $R(\lambda)=\E[\nu f_0]/\E[f_0]$. Then $v=(R-\lambda)/(R+\lambda)$, and $v$ is strictly decreasing on an interval if and only if $x(\lambda)=R(\lambda)/\lambda$ is.
\end{lemma}

\begin{lemma}\label{lem:dq}
For $\lambda_1<\lambda_2$ couple $\beta(\lambda_1)$ and $\beta(\lambda_2)$ on the same trees, write $R_i=R(\lambda_i)$ and $\Delta f_0=f_0(\lambda_2)-f_0(\lambda_1)$. Then
\begin{equation}\label{eq:dq}
x(\lambda_2)<x(\lambda_1)\iff \E\!\left[(\nu-R_1)\frac{\Delta f_0}{\lambda_2-\lambda_1}\right] < \frac{R_1}{\lambda_1}\,\E\bigl[f_0(\lambda_2)\bigr].
\end{equation}
\end{lemma}

Both proofs are short algebra (Appendix~\ref{app:proofs}). The coupling is legitimate because $R(\lambda_1)$ and $R(\lambda_2)$ depend only on the marginal laws.

\section{Bounds on the conductance}\label{sec:conductance}

From here on $\nu$ is uniform on $\{2,3\}$.

\begin{lemma}\label{lem:support}
For $\lambda<2$, almost surely $1-\lambda/2\le\beta\le1-\lambda/3$.
\end{lemma}

\begin{lemma}\label{lem:mono}
On every tree $\lambda\mapsto\beta(\lambda)$ is nonincreasing; hence the law of $\beta(\lambda)$ is stochastically nonincreasing in $\lambda$.
\end{lemma}

\begin{lemma}\label{lem:lip}
For $\lambda_1<\lambda_2<2$, on every tree,
\begin{equation}\label{eq:lip}
0\le\beta(\lambda_1)-\beta(\lambda_2)\le(\lambda_2-\lambda_1)\,\frac{\beta(\lambda_1)}{2-\lambda_2}.
\end{equation}
\end{lemma}

Lemma~\ref{lem:lip} is proved by induction on the finite-depth values $\beta_n$, which are rational functions of $\lambda$, followed by dominated convergence; it is the $m_1=2$ analogue of the bound $|\beta'|\le\beta/(1-\lambda)$ in \citet{aidekon2013note}. It is the only place where the crude constant $1/(2-\lambda)$ enters, and it is what limits the method (Section~\ref{sec:limits}).

\section{Certified bounds on the conductance law}\label{sec:sandwich}

Let $\Tc_\lambda$ map a law $\mu$ on $[a,b]=[1-\lambda/2,1-\lambda/3]$ to the law of $S/(\lambda+S)$, where $S$ is a sum of $\nu$ i.i.d.\ $\mu$-variables. By Lemma~\ref{lem:support}, $\Tc_\lambda$ maps laws on $[a,b]$ to laws on $[a,b]$: the endpoints are fixed points of the binary and ternary recursions.

\begin{lemma}\label{lem:env}
$\Tc_\lambda$ is monotone for the stochastic order $\le_{\mathrm{st}}$. If $U_0\ge_{\mathrm{st}}\beta$ and each $U_{k+1}$ is obtained from $\Tc_\lambda(U_k)$ by moving mass upward, then $U_k\ge_{\mathrm{st}}\beta$ for all $k$; symmetrically for lower envelopes $L_k$ started at $L_0\le_{\mathrm{st}}\beta$.
\end{lemma}

We use a uniform grid of $K+1$ points on $[a,b]$, start from $U_0=\delta_b$ and $L_0=\delta_a$, and after each application of $\Tc_\lambda$ round the image of every lattice value up (for $U$) or down (for $L$) to the grid. Grid points and $\lambda$ are exact rationals; a rounding decision is taken in floating point only when the computed position is at least $10^{-9}$ from a grid point, and otherwise in exact rational arithmetic. Masses are propagated in floating point by direct convolution; after each step the cumulative distribution function is moved in the safe direction by a relative $10^{-10}$ plus an absolute $10^{-10}$, far above the accumulated rounding error (Appendix~\ref{app:float}). No convergence of the iteration is needed: every iterate is a valid envelope. We write $U_\lambda$ and $L_\lambda$ for the last iterates computed at bias $\lambda$.

\section{The cell certificate}\label{sec:cell}

Fix a cell $[\lambda_a,\lambda_b]\subset(1,2)$, let $\Delta=\lambda_b-\lambda_a$, $c=1/(2-\lambda_b)$, $\kappa=1-c\Delta$, $\alpha=\lambda_a-1$, and
\begin{equation}\label{eq:h}
h_3(y,t)=\frac{y\,(ct-1)}{(\alpha+\kappa(y+t))^2},\qquad
h_2(y,t)=\frac{y\,\bigl(1+c(\lambda_b-1+t)\bigr)}{(\alpha+\kappa(y+t))^2}.
\end{equation}

\begin{lemma}\label{lem:worst}
For $\lambda_a\le\lambda_1<\lambda_2\le\lambda_b$, with $\beta_0$ and $T=\sum_{i=1}^{\nu}\beta_i$ taken at $\lambda_1$,
\[
(\nu-R_1)\frac{\Delta f_0}{\Delta\lambda}\le
\begin{cases}(3-R_1)\,h_3(\beta_0,T), & \nu=3,\\ (R_1-2)\,h_2(\beta_0,T),&\nu=2,\end{cases}
\]
provided $\kappa>0$ and $3(1-\lambda_b/2)>2-\lambda_a$.
\end{lemma}

The proof applies the mean value theorem to $f_0$ along the segment from $(\lambda_1,\beta(\lambda_1))$ to $(\lambda_2,\beta(\lambda_2))$, bounds each partial derivative using its sign and Lemma~\ref{lem:lip}, and bounds numerators above and denominators below using $\kappa\beta_j(\lambda_1)\le\beta_j(\lambda_2)\le\beta_j(\lambda_1)$ (Appendix~\ref{app:proofs}).

On the rectangle $y\in[A,B]=[1-\lambda_b/2,\,1-\lambda_a/3]$, $t\in[\nu A,\nu B]$, the function $h_3$ is nondecreasing in both arguments and $h_2$ is nondecreasing in $y$ and nonincreasing in $t$; each condition is affine in $(y,t)$, so it suffices to check it at the worst corner, which the program does in exact arithmetic. By Lemmas~\ref{lem:mono} and~\ref{lem:env}, $\beta(\lambda_1)$ lies between $L_{\lambda_b}$ and $U_{\lambda_a}$ in the stochastic order, so
\[
\E[h_3]\le\E_{U_{\lambda_a},U_{\lambda_a}}[h_3],\qquad \E[h_2]\le\E_{U_{\lambda_a},L_{\lambda_b}}[h_2].
\]
\paragraph{The bracket on $R$.} Write $E_\nu(\lambda)=\E[f_0\mid\nu]$, so $R=2+E_3/(E_2+E_3)$, which is increasing in $E_3$ and decreasing in $E_2$. For $\lambda>1$, $f_0=\beta_0/(\lambda-1+\beta_0+T)$ is increasing in $\beta_0$ and decreasing in $T=\sum_{i=1}^{\nu}\beta_i$ and in $\lambda$, and $\beta_0$ is independent of $T$. So by Lemmas~\ref{lem:mono} and~\ref{lem:env}, for every $\lambda\in[\lambda_a,\lambda_b]$ we have $E_\nu^{\mathrm{lo}}\le E_\nu(\lambda)\le E_\nu^{\mathrm{hi}}$. Here $E_\nu^{\mathrm{hi}}$ takes $\beta_0\sim U_{\lambda_a}$, $T$ distributed as a lower envelope of the $\nu$-fold convolution of $L_{\lambda_b}$, and $\lambda=\lambda_a$; $E_\nu^{\mathrm{lo}}$ takes $\beta_0\sim L_{\lambda_b}$, $T$ from $U_{\lambda_a}$, and $\lambda=\lambda_b$. Set
\[
R_{\mathrm{lo}}=2+\frac{E_3^{\mathrm{lo}}}{E_2^{\mathrm{hi}}+E_3^{\mathrm{lo}}},\qquad
R_{\mathrm{hi}}=\min\Bigl\{m,\;2+\frac{E_3^{\mathrm{hi}}}{E_2^{\mathrm{lo}}+E_3^{\mathrm{hi}}}\Bigr\},\qquad
\E[f_0]_{\mathrm{lo}}=\tfrac12\bigl(E_2^{\mathrm{lo}}+E_3^{\mathrm{lo}}\bigr).
\]
The cap is valid because $R\le m$: given the $\beta_i$, $f_0$ is nonincreasing in $\nu$, and $\nu$ is independent of them, so $\E[\nu f_0]\le m\,\E[f_0]$ by Chebyshev's association inequality. Then $R(\lambda)\in[R_{\mathrm{lo}},R_{\mathrm{hi}}]$ and $\E[f_0(\lambda)]\ge\E[f_0]_{\mathrm{lo}}$ for every $\lambda$ in the cell.

\begin{proposition}\label{prop:cell}
If
\begin{equation}\label{eq:cell}
\tfrac12(3-R_{\mathrm{lo}})\,\E_{U_{\lambda_a},U_{\lambda_a}}[h_3]+\tfrac12(R_{\mathrm{hi}}-2)\,\E_{U_{\lambda_a},L_{\lambda_b}}[h_2]\;<\;\frac{R_{\mathrm{lo}}}{\lambda_b}\,\E[f_0]_{\mathrm{lo}},
\end{equation}
together with the corner conditions, then $v$ is strictly decreasing on $[\lambda_a,\lambda_b]$.
\end{proposition}

\begin{proof}
The left side bounds the left side of \eqref{eq:dq} for every $\lambda_a\le\lambda_1<\lambda_2\le\lambda_b$ (Lemma~\ref{lem:worst}), and the right side bounds its right side from below. Apply Lemmas~\ref{lem:dq} and~\ref{lem:R}.
\end{proof}

\section{Results}\label{sec:results}

We ran the certificate on cells of width $0.01$ covering $[1.17,1.80]$ with $K=1000$, and on cells of width $0.005$ covering $[1.73,1.80]$ with $K=2000$. Table~\ref{tab:cells} lists selected cells; Figure~\ref{fig:margin} shows all of them. Every corner condition held. Rounding decisions near grid points were settled exactly in 308 cases for the first run and 66 for the second.

\begin{table}[h]
\centering
\begin{tabular}{@{}lccc@{}}
\toprule
Cell & $K$ & margin (RHS minus LHS of \eqref{eq:cell}) & certified \\
\midrule
$[1.17,1.18]$ & 1000 & $0.41711$ & yes \\
$[1.40,1.41]$ & 1000 & $0.24765$ & yes \\
$[1.60,1.61]$ & 1000 & $0.11398$ & yes \\
$[1.73,1.74]$ & 1000 & $0.00916$ & yes \\
$[1.74,1.75]$ & 1000 & $-0.00085$ & no \\
$[1.750,1.755]$ & 2000 & $0.00476$ & yes \\
$[1.755,1.760]$ & 2000 & $-0.00032$ & no \\
\bottomrule
\end{tabular}
\caption{Selected cells. All 57 cells of width 0.01 from $1.17$ to $1.74$ are certified, and the refined run certifies every cell from $1.73$ to $1.755$.}
\label{tab:cells}
\end{table}

\begin{figure}[!ht]
\centering
\includegraphics[width=0.85\linewidth]{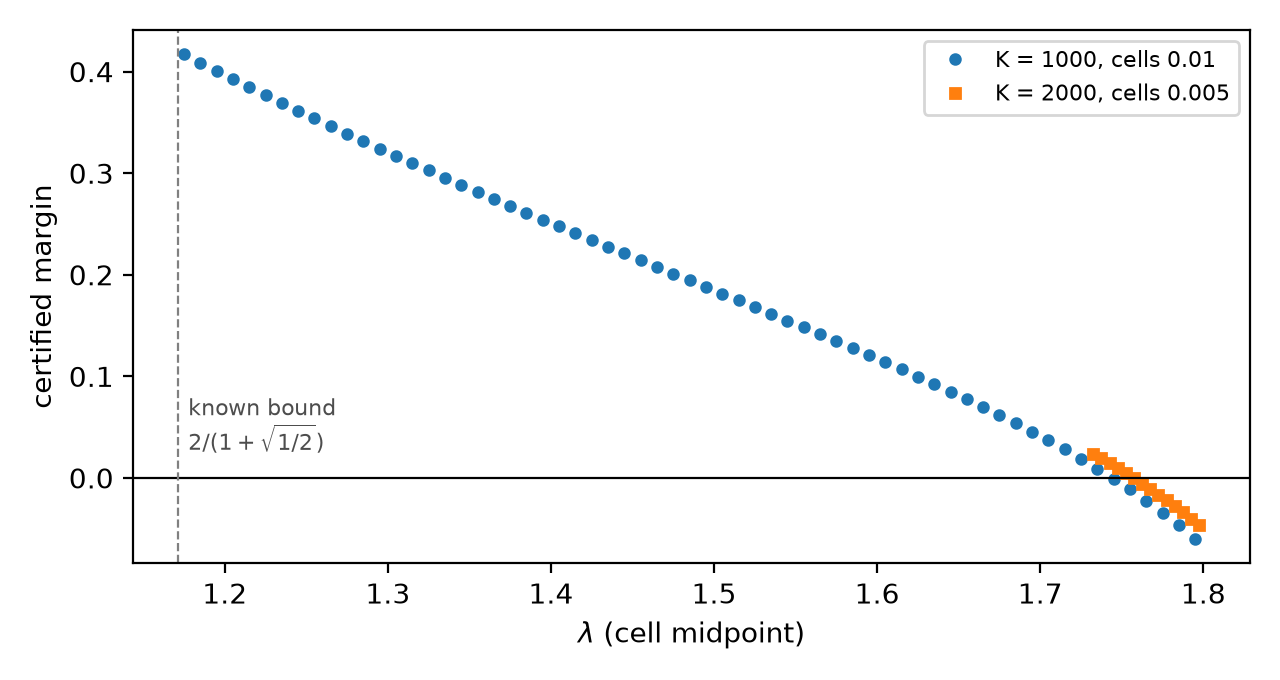}
\caption{Certified margin per cell: the right side minus the left side of \eqref{eq:cell}, plotted at each cell's midpoint. A positive margin certifies the cell. The dashed line marks the bound of \citet{swx2025}.}
\label{fig:margin}
\end{figure}

\paragraph{Independent check.} A second implementation re-derives every step from the lemmas and shares no code with the first: interval arithmetic in mpmath \citep{mpmath}, index decisions taken from the conservative end of each interval, and envelopes defined by their distribution functions. On a coarser grid ($K=150$) it certifies the cells $[1.17,1.18]$, $[1.40,1.41]$ and $[1.60,1.61]$ with margins $0.41463$, $0.24416$ and $0.10837$, slightly below those in Table~\ref{tab:cells}, as a coarser grid should give. At $K=1000$ it certifies $[1.730,1.735]$ with margin $0.02328$, and at $K=2000$ the last cell, $[1.750,1.755]$, with margin $0.00476$; there the two implementations' margins agree to within $10^{-9}$.

\paragraph{Formal verification in Lean~4.} To eliminate both transcription risk in the reduction lemmas and reliance on IEEE-754 floating-point error bounds, we also formalised the argument on $[1.17,1.755]$ in Lean~4 with mathlib \citep{moura2021lean4,mathlib2020}. Every analytical step (Lemmas~\ref{lem:R}--\ref{lem:worst} and Proposition~\ref{prop:cell}) is proved without \texttt{sorry}. Rather than replaying multi-step floating-point iterations inside the theorem prover, the formalisation certifies a single post-fixpoint step per cell endpoint: an integer mass vector on the $(K+1)$-point grid whose rounded push-forward under $\Tc_\lambda$ stochastically dominates (for $U_\lambda$, started from $\delta_a$) or is dominated by (for $L_\lambda$, started from $\delta_b$) itself, propagated to $\beta(\lambda)$ by induction on tree depth and dominated convergence. Across $34$ contiguous rational cells covering $[117/100,351/200]$ (grid sizes $K\le 160$), all envelope, corner and cell inequalities \eqref{eq:cell} are discharged in pure natural-number arithmetic by Lean's kernel reducer (\texttt{Eq.refl true}, without \texttt{native\_decide}). Only standard Lean axioms (\texttt{propext}, \texttt{Classical.choice}, \texttt{Quot.sound}) are used; Aïdékon's speed formula~\eqref{eq:speed} together with the Galton--Watson root-splitting law and \citet[Theorem 1.2]{swx2025} on $[0,1.17]$ enter explicitly as fields of the model hypothesis.

\begin{proof}[Proof of Theorem~\ref{thm:main}]
The certified cells cover $[1.17,1.755]$, and on each closed cell $v$ is strictly decreasing (Proposition~\ref{prop:cell}). A function that is strictly decreasing on each of finitely many closed intervals, each sharing at least a point with the next, is strictly decreasing on their union. On $[0,1.17]\subset[0,2/(1+\sqrt{1/2})]$ the result is \citet[Theorem 1.2]{swx2025}, and $[0,1.17]$ shares the point $1.17$ with the first cell.
\end{proof}

\FloatBarrier
\section{Related work}\label{sec:related}

The speed formula~\eqref{eq:speed} and its derivative at small bias are due to \citet{aidekon2014,aidekon2013note}; \citet{swx2025} extend the small-bias argument using the minimum offspring. On regular trees, \citet{songliu2026} compute the speed and the spectral radius explicitly and show both are monotone in the bias. \citet{bafs2014} treat high bias in their normalisation, which is small $\lambda$ here. The Einstein relation \citep{bhoz2013} and the differentiability result of \citet{bowditch2019} describe $v$ near and on the whole ballistic range; \citet{benarousfribergh2016} survey trapping and speed questions. Non-monotone speeds on trees that are not Galton--Watson are recorded in \citet{lpp1997}. Computer-assisted, rigorous bounds for other quantities of random structures include the confidence intervals of \citet{riordanwalters2007} for percolation thresholds; we are not aware of such certificates for the speed of a random walk. Floating-point error bounds follow \citet{higham2002}; the main program uses NumPy \citep{harris2020numpy}, and the formal proof uses Lean~4 and mathlib \citep{moura2021lean4,mathlib2020}.

\section{Limitations}\label{sec:limits}

\begin{itemize}[leftmargin=*,itemsep=2pt]
\item \textbf{Partial range and one law.} The theorem covers $[0,1.755]$ of $[0,2.5)$ for a single offspring law.
\item \textbf{The crude Lipschitz constant is the obstacle.} With the constant bound $1/(2-\lambda)$ the margin is already negative on $[1.755,1.760]$ at $K=2000$ (Table~\ref{tab:cells}). Replacing the constant by a bound depending on $\beta$, computed by iterating the recursion for the derivative, pushes an uncertified floating-point version further, but we do not claim it here. Reaching $m$ needs distributional control of the derivative of $\beta$, and for $\lambda\ge2$ a new lower bound on $\beta$, since the binary support bound vanishes there.
\item \textbf{External hypotheses in the formalisation.} While all reduction lemmas and the $[1.17,1.755]$ cell certificates are machine-checked in Lean~4 without floating-point arithmetic, Aïdékon's speed formula~\eqref{eq:speed} (with the Galton--Watson branching law) and the small-bias theorem of \citet{swx2025} on $[0,1.17]$ are assumed as hypotheses rather than formalised from scratch; specialist review of that interface remains valuable.
\end{itemize}

\section{Conclusion}\label{sec:conclusion}

For offspring uniform on $\{2,3\}$ the speed of the biased walk is strictly decreasing on $[0,1.755]$, beyond the known $1.17157\ldots$. The certificate rests on three reductions (a difference-quotient form of the monotonicity criterion, a pathwise Lipschitz bound on the conductance, and stochastic-order envelopes for its law), proved both by hand and in Lean~4, together with per-cell rational and exact-integer inequalities. The pipeline should extend to other bounded offspring laws with minimum offspring at least $2$, with the support bound (Lemma~\ref{lem:support}) and the case analysis of Lemma~\ref{lem:worst} redone for each offspring value; for this law, what limits the reach is the crude Lipschitz constant, not the computation.

\paragraph{Code.} \url{https://github.com/samyama-ai/gw-speed-certificate}; \texttt{./run.sh} regenerates the Python certificates, the independent check and Figure~\ref{fig:margin}, and \texttt{cd lean \&\& lake build} checks the Lean~4 formalisation.

\bibliographystyle{plainnat}
\bibliography{paper24_gw_speed_certificate}

\appendix

\section{Proofs}\label{app:proofs}

\begin{proof}[Proof of Lemma~\ref{lem:R}]
$\E[(\nu\mp\lambda)f_0]=\E[\nu f_0]\mp\lambda\E[f_0]$; divide by $\E[f_0]>0$. The map $u\mapsto(u-1)/(u+1)$ is increasing, and $v=(x-1)/(x+1)$ with $x=R/\lambda$.
\end{proof}

\begin{proof}[Proof of Lemma~\ref{lem:dq}]
Write $N=\E[\nu f_0]$ and $M=\E[f_0]$. Then $N_2M_1-N_1M_2=M_1\bigl(\E[\nu\Delta f_0]-R_1\E[\Delta f_0]\bigr)=M_1\E[(\nu-R_1)\Delta f_0]$, so $R_2-R_1=\E[(\nu-R_1)\Delta f_0]/M_2$. Finally $x(\lambda_2)<x(\lambda_1)$ if and only if $R_2-R_1<R_1(\lambda_2-\lambda_1)/\lambda_1$.
\end{proof}

\begin{proof}[Proof of Lemma~\ref{lem:support}]
$S/(\lambda+S)$ is increasing in $S$, so by induction on depth the finite-depth values are monotone under inclusion of trees with a common root, and so is the limit. Every $\{2,3\}$ tree contains a binary tree and is contained in a ternary tree, On the $d$-ary tree the finite-depth values, started from $1$, decrease to the largest fixed point of $y\mapsto dy/(\lambda+dy)$, which is $1-\lambda/d$ (the other fixed point is $0$).
\end{proof}

\begin{proof}[Proof of Lemma~\ref{lem:mono}]
$S/(\lambda+S)$ is decreasing in $\lambda$ and increasing in $S$; induct on depth and pass to the limit.
\end{proof}

\begin{proof}[Proof of Lemma~\ref{lem:lip}]
Let $\delta_n=-\partial_\lambda\beta_n$. Differentiating the recursion gives $\delta_n=(S+\lambda\sum_i\delta_n(x_i))/(\lambda+S)^2$, and $\delta_n=0$ at depth $n$. If $\delta_n(x_i)\le c\,\beta_n(x_i)$ for the children, then $\delta_n\le(1+\lambda c)S/(\lambda+S)^2=(1+\lambda c)\beta_n/(\lambda+S)$. Since each child is the root of a $\{2,3\}$ tree, $\beta_n(x_i)\ge\beta(x_i)\ge1-\lambda/2$ by Lemma~\ref{lem:support}, so $\lambda+S\ge2$ and $\delta_n\le(1+\lambda c)\beta_n/2\le c\beta_n$ whenever $c\ge1/(2-\lambda)$. Thus $\beta_n(\lambda_1)-\beta_n(\lambda_2)\le\int_{\lambda_1}^{\lambda_2}\beta_n(s)/(2-s)\,ds$. Let $n\to\infty$ by dominated convergence ($0\le\beta_n\le1$), and use $\beta(s)\le\beta(\lambda_1)$ (Lemma~\ref{lem:mono}). The lower bound is Lemma~\ref{lem:mono}.
\end{proof}

\begin{proof}[Proof of Lemma~\ref{lem:env}]
The stochastic order is preserved by independent sums and increasing maps, $\beta$ is a fixed point of $\Tc_\lambda$, and moving mass upward can only increase a law in $\le_{\mathrm{st}}$. Induct on $k$.
\end{proof}

\begin{proof}[Proof of Lemma~\ref{lem:worst}]
By Lemmas~\ref{lem:mono} and~\ref{lem:lip},
\[
\kappa\beta_j(\lambda_1)\le\beta_j(\lambda_2)\le\beta_j(\lambda_1),\qquad \frac{\Delta\beta_j}{\Delta\lambda}\in[-c\beta_j(\lambda_1),0].
\]
The function $f_0(\lambda,y)=y_0/(\lambda-1+\sum_j y_j)$ is $C^1$ for $\lambda>1$, $y\in(0,1]^{\nu+1}$; by the mean value theorem there is a point $\xi$ on the segment with $\Delta f_0=\partial_\lambda f_0(\xi)\Delta\lambda+\sum_j\partial_jf_0(\xi)\Delta\beta_j$. At $\xi$, with $D=\xi_\lambda-1+\sum_j\xi_j$: $\partial_\lambda f_0=-\xi_0/D^2$, $\partial_0f_0=(D-\xi_0)/D^2>0$, $\partial_if_0=-\xi_0/D^2$ for $i\ge1$.

If $\nu=3$, then $w=3-R_1>0$ since $R_1\le m$. The $\Delta\beta_0$ term of $w\,\Delta f_0/\Delta\lambda$ is nonpositive, and each $\Delta\beta_i$ term is at most $w\xi_0c\beta_i(\lambda_1)/D^2$. Hence $w\Delta f_0/\Delta\lambda\le w\xi_0(cT-1)/D^2$, where $cT-1>0$ because $cT\ge3(1-\lambda_b/2)/(2-\lambda_b)=3/2$. Bound $\xi_0\le\beta_0(\lambda_1)$ above and $D\ge\alpha+\kappa(\beta_0(\lambda_1)+T)$ below.

If $\nu=2$, then $w=2-R_1<0$. The $\partial_\lambda$ term is $|w|\xi_0/D^2\le|w|\beta_0(\lambda_1)/D^2$, which gives the $1$ in $h_2$; the $\Delta\beta_i$ terms, $i\ge1$, are nonpositive; the $\Delta\beta_0$ term is at most $|w|(D-\xi_0)c\beta_0(\lambda_1)/D^2$; and $D-\xi_0\le\lambda_b-1+T$. So $w\Delta f_0/\Delta\lambda\le|w|\beta_0(\lambda_1)(1+c(\lambda_b-1+T))/D^2$, with the same lower bound on $D$.
\end{proof}

\section{Floating-point error}\label{app:float}

We use the standard model \citep{higham2002}: an operation whose result does not underflow returns the exact result times $1+\delta$ with $|\delta|\le u=2^{-53}$, and a sum or dot product of $n$ nonnegative terms, in any order, has relative error at most $\gamma_n=nu/(1-nu)$. All masses and integrands are nonnegative. For $K\le2000$ the convolutions, the push-forward to the grid and the cumulative sums give relative errors at most $\gamma_{7K+7}<2\times10^{-12}$. Masses far in a tail can underflow. An operation with a subnormal result errs by at most $2^{-1074}$ in absolute value, and one step uses fewer than $10^{8}$ operations, so the absolute error per step is below $10^{-300}$. The absolute part of the shift covers it: each iterate is checked against the exact image of the previous computed iterate, so errors do not accumulate across steps. Each integrand is a rational function with at most six operations. It is evaluated at floating-point grid points $\mathrm{fl}(a)+\mathrm{fl}(h)\,i$, which differ from the exact ones by a relative $3u$, and its logarithmic sensitivity to each coordinate is at most $2$; with the expectation sums, the relative error is below $\gamma_{4K+2}+20u<10^{-12}$. Every expectation used is at least $10^{-3}$, so underflow inside the sums is negligible against the relative $10^{-10}$ by which each is widened in the safe direction. Rounding positions $x=(t-a)/h$ are computed with absolute error below $10^{-11}$, so a floating-point decision taken only when $x$ is at least $10^{-9}$ from an integer agrees with the exact one; closer cases are decided in exact rational arithmetic. Floating-point numbers convert exactly to rationals, and the final comparison~\eqref{eq:cell} is exact.

\end{document}